\documentclass[12pt,a4paper]{article}
\usepackage{amsfonts}
\usepackage{mathrsfs}
\usepackage{amsmath,amssymb,amsthm}
\usepackage{tikz}
\usepackage{setspace}
\usetikzlibrary{positioning, shapes.geometric}
\usepackage{geometry}
\newtheorem{theorem}{Theorem}[section] 
  
\newtheorem{lemma}[theorem]{Lemma} 
  
\newtheorem{claim}[theorem]{Claim}

\numberwithin{equation}{section}

\usepackage{titlesec}
\titleformat{\section}[block]
{\normalfont\fontsize{16}{19.2}\selectfont\bfseries}
{\thesection}
{1em}
{}

\title{\textbf{On the product of cross-intersecting families with small covering number}}
\author{
	Peter Frankl\textsuperscript{1} , Long Lin\textsuperscript{2} 
	, Hehui Wu\textsuperscript{2}\\ \\
	\textsuperscript{1}R\'enyi Institute, Budapest, Hungary \\
	\textsuperscript{2}Shanghai Center for Mathematical Sciences\\ Fudan University, Shanghai, China\\ \\
}
\date{}

\begin{document}
	
	\maketitle
		\begin{abstract}
	A central problem in extremal set theory is to determine or estimate $m(n,k,t),n>2k\geq 2t$, the maximum size of an intersecting $k$-graph and covering number at least $t$(see the paper for the definitions). For $t=1$ and $2$ the classical Erd\H{o}s-Ko-Rado Theorem and the Hilton-Milner Theorem provide the answer.The complete solution for $t=3$ was only achieved recently . There are some partial results for $t=4,5$ but for the general case even to determine the asymptotic  appears to be hopelessly difficult .
	 
	Denoting by $\widetilde{m}(n,k,t)$ the maximum of $|\mathcal{F}||\mathcal{G}|$ for a pair of cross-intersecting $k$-graphs with covering number at least $t$, $\widetilde{m}(n,k,t)\geq {m}(n,k,t)^2$ is obvious. Pyber showed that equality holds for $t=1$. The same was shown for $t=2$ in a wide range(cf.[7]).
	Quite surprisingly our results show that the inequality is strict for $t\geq 3$ and for $n>n_0(k,t)$, Theorem 1.7 determines the exact value of $\widetilde{m}(n,k,t)$ for $k>2t$ and  $n$ sufficiently large.

	\end{abstract}
	
	\section{Introduction}
	
	Let $[n]=\{1,2,\ldots,n\}$ be the standard $n$-set, $2^{[n]}$ its power set. Subsets of $2^{[n]}$ are called families. For $0\leq k\leq n$, $\binom{[n]}{k}$ is the collection of all $k$-subsets. Subsets of $\binom{[n]}{k}$ are called $k$-graphs.
	
	For a family $\mathcal{F}$, set $\bigcap \mathcal{F}=\displaystyle \bigcap_{F\in \mathcal{F}}F$. If $\bigcap \mathcal{F}\neq \emptyset$ then $\mathcal{F}$ is called a star.
	
		A family $\mathcal{F}$ is called \emph{intersecting} if $F \cap F' \neq \varnothing$ holds for all $F, F' \in \mathcal{F}$. Similarly, two families $\mathcal{F}$ and $\mathcal{G}$ are called \emph{cross-intersecting} (CI for short) if $F \cap G \neq \varnothing$ holds for all $F \in \mathcal{F}$, $G \in \mathcal{G}$.
		
		A set $T$ is called a transversal of the family $\mathcal{F}$ if $\{T\}$ and $\mathcal{F}$ are cross-intersecting. Define the covering number 
		$$\tau(\mathcal{F})=\min\{|T|: T \text{ is called a transversal of }\mathcal{F} \}.$$
		
		Note that if  $\mathcal{F}$ is a star then $\tau(\mathcal{F})=1$; if  $\mathcal{F}$ is an intersecting $k$-graph then $\tau(\mathcal{F})\leq k$. More generally, if a $k$-graph $\mathcal{F}$ and a $\ell$-graph $\mathcal{G}$ are CI then $\tau(\mathcal{F})\leq \ell$ and $\tau(\mathcal{G})\leq k$.
		
		Let us recall two fundamental theorems of extremal set theory.
		
		\noindent \textbf {Theorem 1.1} (Erd\H{o}s-Ko-Rado [1]) Let $n \geq 2k > 0$ and suppose that $\mathcal{F} \subset \binom{[n]}{k}$ is intersecting. Then
			\begin{equation}
				|\mathcal{F}| \leq \binom{n-1}{k-1}.
			\end{equation}

	\noindent \textbf {Theorem 1.2}
	(Erd\H{o}s-Lov\'asz [2]) Suppose that $\mathcal{F}$ is an intersecting $k$-graph satisfying $\tau(\mathcal{F})=k$. Then
	\begin{equation}
		|\mathcal{F}| \leq k^k.
	\end{equation}

	As to (1.1), the full star $\{F \in \binom{[n]}{k} : 1 \in F\}$ shows that the bound is best possible.
	
	The inequality (1.2) justifies the following definition:
	\[
	m(k) := \max \left\{ |\mathcal{F}| : \mathcal{F} \text{ is an intersecting } k\text{-graph with } \tau(\mathcal{F}) = k \right\} .
	\]
	
	It is known that $m(2)=3, m(3)=10$ (cf. [2]) but the determination of $m(k)$ for $k \geq 4$ appears to be very difficult. Already for $k=4$ there is a large gap between the lower and upper bounds(cf. [5] and [6]):
	\begin{equation}
		42 \leq m(4) \leq 64 \tag{1.3}.
	\end{equation}

	For fixed $1 \leq s \leq k$ let us define
	\[
	m(n,k,s) = \max \left\{ |\mathcal{F}| : \mathcal{F} \subset \binom{[n]}{k} ; \mathcal{F} \text{ is intersecting and } \tau(\mathcal{F}) \geq s \right\}.
	\]
	By the Erd\H{o}s-Ko-Rado Theorem,
	\[
	m(n,k,1) = \binom{n-1}{k-1}.
	\]
	Hilton and Milner [9] determined $m(n,k,2)$ for $n > 2k$:
	\begin{equation}
		m(n,k,2) = \binom{n-1}{k-1} - \binom{n-k-1}{k-1} + 1 = k\binom{n-2}{k-2} + O\left(\binom{n-3}{k-3}\right). \tag{1.4}
	\end{equation}

The first author determined $m(n,k,3)$ for $n > n_0(k)$. In [8] by a different proof this was extended to the full range $n > 2k$:
	$$ \quad m(n,k,3) = \binom{n-1}{k-1} - \binom{n-k}{k-1} - \binom{n-k-1}{k-1} + \binom{n-2k}{k-1} + \binom{n-k-2}{k-3} + 3  =$$
	
	\begin{equation}
		(k^2-k+1)\binom{n-3}{k-3}+O\left(\binom{n-4}{k-4}\right) . \tag{1.5}
	\end{equation}

As to the general case, it seems to be hopeless. Let us mention an old conjecture concerning the asymptotics.

\noindent \textbf{Conjecture 1.3} ([12]) For $s \geq 3$, $k > k_0(s)$, $n \geq n_0(k,s)$:
\begin{equation}
	 \quad m(n,k,s) = \left(k^{s-1} - \binom{s-1}{2}k^{s-2} + O(k^{s-3})\right)\binom{n-s}{k-s}.\tag{1.6}
\end{equation}

The present paper is devoted to the case of two cross-intersecting families . Quite surprisingly we can get some exact bounds.

Let us consider now the case of a $k$-graph $\mathcal{F}$ and a $\ell$-graph $\mathcal{G}$ that are CI. Define
$\widetilde{m}(n,k,\ell,s,t)$ is the maximum of $|\mathcal{F}||\mathcal{G}|$ subject to $\mathcal{F}$  is a $k$-graph with $\tau(\mathcal{F}) = s $, $\mathcal{G}$ is a $\ell$-graph with  $\tau(\mathcal{G}) =t$  and they are CI .

Let us recall some known results.

\noindent \textbf{Theorem 1.4} (Pyber[3], Matsumoto-Tokushige[4])
\begin{equation}
	\widetilde{m}(n,k,\ell,1,1) = \binom{n-1}{k-1}\binom{n-1}{\ell-1} \quad \text{ for } n \geq 2k \geq 2\ell \geq 2.\tag{1.7}
\end{equation}

\noindent \textbf{Theorem 1.5} (Frankl-Wang [7] and [11] )
\begin{equation}
	\widetilde{m}(n,k,k,2,2) = \left( \binom{n-1}{k-1} - \binom{n-k-1}{k-1} + 1 \right)^2 \quad \text{for } n \geq 2k \geq 18.\tag{1.8}
\end{equation}

Let us mention that it is conjectured that (1.8) holds for all $n \geq 2k , 4\leq k\leq 8$ as well. In [13] (1.8) is established for $k=3, n \geq 6$.

The main result of the present paper is the determination of $\widetilde{m}(n,k,\ell,s,t)$ for all quadruples $(k,\ell,s,t)$ with $k >  2t \geq 1$, $\ell > 2s \geq 1$ and $n > n_0(k,\ell,s,t)$.

Let us first describe the construction.

\noindent \textbf{Example 1.6}
	Let $k,\ell,s,t$ be integers, $n \geq k\ell$, $k \geq t \geq 1$, $\ell \geq s \geq 1$. Let $F_1, F_2, \dots, F_{s-1}$ and $G_1, \dots, G_{t-1}$ be pairwise disjoint $k$-sets and $\ell$-sets in $[2,n]$, respectively. Assume that $\mathcal{F}_0 := \{F_1, \dots, F_{s-1}\}$ and $\mathcal{G}_0 := \{G_1, \dots, G_{t-1}\}$ are CI. Define
	\[
	\mathcal{P} = \{1\} \times G_1 \times \dots \times G_{t-1}, \quad \mathcal{R} = \{1\} \times F_1 \times \dots \times F_{s-1}.
	\]
	Note that $\mathcal{P}$ is an $t$-graph, $\mathcal{R}$ is a $s$-graph. Define
	\[
	\widetilde{\mathcal{F}} = \mathcal{F}_0 \cup \{F \in \binom{[n]}{k} : \exists P \in \mathcal{P}, P \subset F\},
	\]
	\[
	\widetilde{\mathcal{G}} = \mathcal{G}_0 \cup \{G \in \binom{[n]}{\ell} : \exists R \in \mathcal{R}, R \subset G\}.
	\]
	It is easy to check that $\mathcal{F}, \mathcal{G}$ are CI, $\tau(\mathcal{F})=s, \tau(\mathcal{G})=t$. For $k > t$,
	\[
	|\widetilde{\mathcal{F}}| = \ell^{t-1} \binom{n-t}{k-t} + O(n^{k-t-1}),
	\]
	for $\ell > s$,
	\[
	|\widetilde{\mathcal{G}}| = k^{s-1} \binom{n-s}{\ell-s} + O(n^{\ell-s-1}).
	\]
	
	We should point out that both $\widetilde{\mathcal{F}}$ and $\widetilde{\mathcal{G}}$ are determined by the construction up to isomorphism. However for the intersection pattern of the pair $(\mathcal{F}_0, \mathcal{G}_0)$ there are a growing number of possibilities as $k,l \to \infty$. That is, for the pair $(\mathcal{F}, \mathcal{G})$ there are many non-isomorphic choices.
	
	Now we can state our main result.

\noindent \textbf{Theorem 1.7} 
	Let $k,\ell,s,t$ be integers, $k > 2t \geq 1$, $\ell > 2s \geq 1$. Suppose that $\mathcal{F} \subset \binom{[n]}{k}, \mathcal{G} \subset \binom{[n]}{\ell}$ are CI with $\tau(\mathcal{F})=s, \tau(\mathcal{G})=t$. Then for $n > n_0(k,\ell)$,
	\begin{equation}
		 \quad |\mathcal{F}| |\mathcal{G}| \leq |\widetilde{\mathcal{F}}| |\widetilde{\mathcal{G}}|. \tag{1.9}
	\end{equation}
	Moreover, in case of equality $\mathcal{F}$ is isomorphic to $\widetilde{\mathcal{F}}$ and $\mathcal{G}$ is isomorphic to $\widetilde{\mathcal{G}}$.
	
\section{Bounding $\mathcal{F}$ and $\mathcal{G}$ via a branching process}

Let us start with some important observations.

Since we are trying to prove upper bounds on $|F||G|$, we may assume that the CI pair $\mathcal{F} \subset \binom{[n]}{k}$ and $\mathcal{G} \subset \binom{[n]}{\ell}$ is maximal, that is, for $F \in \binom{[n]}{k} \setminus \mathcal{F}$ and $G \in \binom{[n]}{\ell} \setminus \mathcal{G}$ neither $(\mathcal{F} \cup \{F\}, \mathcal{G})$ nor $(\mathcal{F}, \mathcal{G} \cup \{G\})$ is CI. Let us note the trivial fact $\tau(\mathcal{H}) \leq \tau(\mathcal{H}')$ whenever $\mathcal{H} \subset \mathcal{H}'$, i.e., making the CI pair $(\mathcal{F}, \mathcal{G})$ maximal will not lower the covering numbers.

Define $\mathcal{F}^* = \{H \subset [n] : |H| \leq k, H \text{ is a transversal of } \mathcal{G}\}$ and define $\mathcal{G}^*$ analogously.

\begin{lemma}
	\begin{enumerate}
		\item[(i)] If $B \in \mathcal{F}^*$ and $B \subset F \in \binom{[n]}{k}$ then $F \in \mathcal{F}$,
		\item[(ii)] If $C \in \mathcal{G}^*$ and $C \subset G \in \binom{[n]}{\ell}$ then $G \in \mathcal{G}$,
		\item[(iii)] $\mathcal{F}^*$ and $\mathcal{G}^*$ are CI for $n \geq k+\ell$.
	\end{enumerate}
\end{lemma}

\begin{proof}
	To prove (i) just note that $B \in \mathcal{F}^*$ implies that $\mathcal{F} \cup \{F\}$ and $\mathcal{G}$ are CI. As the pair $(\mathcal{F}, \mathcal{G})$ is maximal, $F \in \mathcal{F}$.
	The same argument works for (ii).
	
	Let us prove (iii) indirectly. Suppose that ${F}^*\in \mathcal{F}^*$ and ${G}^*\in \mathcal{G}^*$ are disjoint. Using $n \geq k+\ell$,  fix $F \in \binom{[n]}{k}$ and $G \in \binom{[n]}{\ell}$ with $F \cap G = \emptyset$ and still $F^* \subset F$, $G^* \subset G$. By (i) and (ii), $F \in \mathcal{F}$, $G \in \mathcal{G}$. This contradicts the fact that $(\mathcal{F}, \mathcal{G})$ form a CI pair .
\end{proof} 

Let $\mathcal{B}(\mathcal{F}^*) \ (\mathcal{B}(\mathcal{G}^*))$ be the collection of containment-minimal members of $\mathcal{F}^* \ (\mathcal{G}^*)$, respectively. In view of Lemma 2.1, 
\[ \mathcal{F} = \{F \in \binom{[n]}{k} : \exists B \in \mathcal{B}(\mathcal{F}^*) : B \subset F \} \text{ and similarly for } \mathcal{G}.\]
Set $\mathcal{B}(i) = \mathcal{B}(\mathcal{F}^*) \cap \binom{[n]}{i}$ and $\mathcal{B}'(i) = \mathcal{B}(\mathcal{G}^*) \cap \binom{[n]}{i}$. Then $\mathcal{B}(i) = \emptyset$ for $i < t$ whence

\begin{equation}
	|\mathcal{F}| \leq \sum_{t \leq i \leq k} |\mathcal{B}(i)| \binom{n-i}{k-i}.
\end{equation}

Let us now describe a branching process to find the members of $\mathcal{B}(\mathcal{F}^*)$. First fix a set $P_1 \in \mathcal{B}'(s)$ and define $s$ sequences $(x_i), x_i \in P_1$ of length $1$.

Once we have defined the sequences of length $q$ check for each sequence $(x_1, \dots, x_q)$ whether $\{x_1, \dots, x_q\}$ is a transversal of $\mathcal{G}$. If ``yes'' then stop. If ``no'' then choose $P_{q+1} \in \mathcal{B}(\mathcal{G}^*)$ of minimal size with $\{x_1, \dots, x_q\} \cap P_{q+1} = \emptyset$ and define the $|P_{q+1}|$ extended sequences $(x_1, \dots, x_q, x_{q+1}) : x_{q+1} \in P_{q+1}$. If $q=k$ then stop also. Note that $(x_1, \dots, x_q)$ might fail to produce a transversal of $\mathcal{F}$ but it will not affect our upper bounds.

\begin{claim}
	\begin{equation}
		 |\mathcal{B}(j)| \leq s \ell^{j-1} \quad \text{for } t \leq j \leq k. 
	\end{equation}
	
\end{claim}

\begin{proof}
	From the construction of the branching process it is clear that we constructed at most $s \ell^{j-1}$ transversals $\{x_1, \dots, x_j\}$ of size $j$. Thus once we prove that each $B \in \mathcal{B}(j)$ occurs at least once, the proof of (2.2) is done.
	
	Assume the contrary and fix $B \in \mathcal{B}(j)$ for some $t \leq j \leq k$ that does not occur. Let $\{x_1, \dots, x_j\}$ be a maximal subset of $B$ such that the sequence $(x_1, \dots, x_j)$ occurs in the branching process. As a proper subset of $B \in \mathcal{B}(j)$ it cannot be a transversal of $\mathcal{G}$. Hence we extended $(x_1, \dots, x_j)$ using some $P_{j+1} \in \mathcal{B}(\mathcal{G}^*)$ disjoint to $\{x_1, \dots, x_j\}$ (note that for $j=0$, $|P_{j+1}|=s$). As $P_{j+1} \cap B 
	\neq \emptyset$ (cf.\ Lemma~2.1(iii)) we may fix $x_{j+1} \in P_{j+1} \cap B$.
	
	Then $\{x_1, \dots, x_j, x_{j+1}\} \subset B$, contradicting the maximal choice of $j$. This proves that $B$ occurs in the branching process. Thus (2.2) holds .
\end{proof} 

Let us stress that $t = \min\{i : B(i) 
\neq \emptyset\}$ and 
$s = \min\{j : B'(j) 
\neq \emptyset\}$. We will use the formulas (2.3) :
\begin{align*}
	|\mathcal{F}| &\leq |B(t)|\binom{n-t}{k-t} + \sum_{t < i \leq k} |B(i)|\binom{n-i}{k-i} = |B(t)|\binom{n-t}{k-t} + O(n^{k-t-1}), \\   
	|\mathcal{G}| &\leq |B'(s)|\binom{n-s}{\ell-s} + \sum_{s < j \leq \ell} |B'(j)|\binom{n-j}{\ell-j} = |B'(s)|\binom{n-s}{\ell-s} + O(n^{\ell-s-1}). 
\end{align*}
Hence the asymptotic size of $|\mathcal{F}||\mathcal{G}|$ depends only on 
$|\mathcal{B}(t)||\mathcal{B}'(s)|$.

\medskip

 Set $\mathcal{P} = \mathcal{B}(t)$, $u = \tau(\mathcal{P})$, $\mathcal{R} = \mathcal{B}'(s)$, $v = \tau(\mathcal{R})$.
 
\begin{lemma} 
 	(i) $|\mathcal{P}| \leq u s^{v-1} \ell^{t-v}$, (ii) $|\mathcal{R}| \leq v t^{u-1} k^{s-u}$.
 
\end{lemma}

 \begin{proof}
 	 By symmetry let us only prove (i).
 
 First fix a $u$-element set $A$ which is transversal to $\mathcal{P}$. Then construct a sequence $(x_i)$ of length $1$ for each $x_1 \in A$.
 
 If $1 \leq i < v$ and $(x_1, \dots, x_i)$ is among the sequences of length $i$ then fix $R_i \in \mathcal{R}$ with $\{x_1, \dots, x_i\} \cap R_i = \emptyset$. It is possible by $\tau(R)=v$. Then extend the sequence to $(x_1, \dots, x_i, x_{i+1})$ in all possible $s$ ways where $x_{i+1} \in R_i$.
 
 For $v \leq i < t$ and to every sequence $(x_1, \dots, x_i)$ fix $G_i \in \mathcal{G}$ with $\{x_1, \dots, x_i\} \cap G_i = \emptyset$. It is possible by $\tau(\mathcal{G})=t$. Then extend the sequence to $(x_1, \dots, x_i, x_{i+1})$ in all possible $l$ ways where $x_{i+1} \in R_i$.
 
 Eventually we construct $u s^{v-1} \ell^{t-v}$ sequences. Some are not transversals of $\mathcal{G}$, different sequences might correspond to the same $t$-set but these do not affect the upper bound (i) as long as each $P \in P$ occurs at least once. This fact can be proved in the very same way as for Claim 2.2 .
 \end{proof}

\section{Proof of Theorem 1.7}

    First let us recall the following classical result of Gyárfás [8] :

\begin{theorem}\label{thm:gyarfas}
	For any $k$-graph $\mathcal{F}$ with $\tau(\mathcal{F}) = t$,
	\begin{equation}
		|\mathcal{T}(\mathcal{F})| \leq k^t
	\end{equation}
	with equality if and only if $\mathcal{F}$ consists of $t$ pairwise disjoint edges.
\end{theorem}
Let $k,l,s,t$ be integers, $k > 2t \geq 1$, $\ell > 2s \geq 1$.
 
 \begin{claim}
  
  $u t^{u-1} k^{s-u} \leq k^{s-1}$ and $v s^{v-1} \ell^{t-v} \leq \ell^{t-1}$ with equality iff $u=1$ and $v=1$ , respectively.

 \end{claim}
 
 \begin{proof}  Since $(\frac{k}{t})^{u-1}\geq 2^{u-1} \geq u$, $u t^{u-1} k^{s-u} \leq k^{s-1}$ follows and equality holds iff $u=1$ . Similarly for $(v,s,\ell)$.
\end{proof}

   Now we are able to prove Theorem 1.7 :
 \begin{proof}
 	We proved $|\mathcal{P}||\mathcal{R}|\leq k^{s-1}\ell^{t-1}$.If the inequality is strict then $|\mathcal{F}||\mathcal{G}|<|\widetilde{\mathcal{F}}||\widetilde{\mathcal{G}}|$ follows for $n>n_0(k,\ell)$ using (2.3). Hence we may suppose $|\mathcal{P}||\mathcal{R}|=k^{s-1}\ell^{t-1}$.
 	By Lemma 2.3 and Claim 3.2 , $u=v=1$ and $|\mathcal{P}| =\ell^{t-1},|\mathcal{R}| =k^{s-1}$. Let $x\in \bigcap{\mathcal{P}}$,  then $\tau(\mathcal{G}(\overline{x}))=t-1$ and $\mathcal{P}=\{x\}\times \mathcal{T}(\mathcal{G}(\overline{x}))$. As $|\mathcal{T}(\mathcal{G}(\overline{x}))|=|\mathcal{P}|=\ell^{t-1}$, the uniqueness in Theorem 3.1 implies $\mathcal{G}(\overline{x})$ consists of $t-1$ pairwise disjoint edges. Similarly let $y\in \bigcap{\mathcal{R}}$,  then $\mathcal{R}=\{y\}\times \mathcal{T}(\mathcal{F}(\overline{y}))$ and $\mathcal{F}(\overline{y})$ consists of $s-1$ pairwise disjoint edges. Since $\mathcal{P} , \mathcal{R}$ are CI, $x=y$ follows. Thus $\mathcal{F}$ is isomorphic to $\widetilde{\mathcal{F}}$, $\mathcal{G}$ is isomorphic to $\widetilde{\mathcal{G}}$, completing the proof.
 \end{proof}
    With some extra effort we succeeded in extending the range of Theorem 1.7 to $k=2t$ and $\ell=2s$. However to prove the statements for all $k>t, \ell >s$ seems to need some new ideas.

\end{document}